\documentclass[12pt,a4paper]{amsart}
\usepackage{amsmath}
\usepackage{amssymb}
\usepackage{amsthm}
\usepackage{amsfonts}
\usepackage[nobysame,abbrev,alphabetic]{amsrefs}
\usepackage[utf8]{inputenc}
\usepackage{mathtools}
\usepackage{ifpdf}
\ifpdf \usepackage[pdftex,pdfstartview=FitH,pdfpagemode=none,colorlinks,bookmarks,linkcolor=blue]{hyperref} \else  \usepackage[hypertex]{hyperref} \fi

\newtheorem{theorem}{Theorem}
\newtheorem{lemma}[theorem]{Lemma}
\newtheorem{remark}[theorem]{Remark}

\newtheorem{assumption}[theorem]{Assumption}
\newtheorem{corollary}[theorem]{Corollary}

\newcommand{\tv}{{\tilde {v}}}
\newcommand{\tA}{{\tilde {A}}}

\newcommand{\T}{{\mathrm{T}}}

\newcommand{\bZ}{{\mathbb {Z}}}

\begin{document}
\bibliographystyle{plain}
\title{Local rigidity of certain solvable group actions on tori}
\author[Qiao Liu]{Qiao Liu}
\address{Pennsylvania State University, State College, PA 16802, USA}
\setcounter{page}{1}
\begin{abstract}
In this paper, we study a local rigidity property of $\mathbb Z \ltimes_\lambda \mathbb R$ affine action on tori generated by an irreducible toral automorphism and a linear flow along an eigenspace. Such an action exhibits a weak version of local rigidity, i.e., any smooth perturbations close enough to an affine action is smoothly conjugate to the affine action up to proportional time change.
\end{abstract}
\maketitle
{\small\tableofcontents}
\section{Introduction}

Considering the solvable group $\mathbb Z\ltimes_\lambda\mathbb R $ acting on torus $\mathbb T^d$ by an irreducible toral automorphism  and a linear flow,   we can ask whether this action carries local rigidity property. More precisely, given a smooth small perturbation, is it smoothly conjugate to the initial system? The problem interests us from following two aspects.

In hyperbolic dynamics, Anosov diffeomorphism are structurally stable, meaning that any small $C^1$ perturbation is topologically conjugate to the initial system. However the conjugacy is in general only bi-H\"older. Some necessary conditions for having smooth conjugacy have been studied regarding their sufficiency. One such condition is that two smoothly conjugate Anosov diffeomorphism must have the same periodic data (eigenvalues of derivatives at periodic points). De la Llave, Marco and Moriy\'on \cite{de1986canonical}
{}give a positive answer to this question in dimension 2. In higher dimension, there is a counterexample constructed by de la Llave \cite{de1992smooth} showing that one has to add extra assumption such as irreducibility of the linear map, real simple spectrum, etc.  Along this direction, there are many contribution by de la Llave, Gogolev, Kalinin, Sadovskaya and many others. Another necessary condition of smooth conjugacy between an Anosov diffeomorphism $f$ and its linear part $A$ is having the same Lyapunov exponents. Saghin and Yang \cite{saghin2018lyapunov} obtained smoothness of the conjugacy for a volume preserving perturbation $f$ of an irreducible $A$ assuming they have the same simple Lyapunov exponents. See also the recent work \cite{gogolev2018local}, where the assumption on the simplicity of Lyapunov exponents is relaxed.

For our group action,  any small enough smooth perturbation still contains a smooth flow. The flow orbits give a natural foliation and it is invariant under the diffeomorphism. Rigidity of perturbation of linear flow on torus is  a classic problem. For circle diffeomorphisms, the rotation number gives a complete topological classification by the classical Denjoy's theorem. Moreover regularity of conjugacy depends on arithmetical properties of rotation number due to their crucial role in small divisor problems. This was due to the works of  Arnol'd, Moser, Kolmogorov, Herman and Yoccoz and many others. In higher dimensional tori, there does not have to be an unique rotation vector  and different orbits may have different asymptotic behaviors. Under assumption that all orbits with the same rotation vectors, local rigidity was extended to higher dimension by KAM method.  For instance, Dias \cite{dias2008local} showed real-analytic perturbations of a linear flow on a torus is analytically conjugate to linear flow if the rotation vector satisfies a more general Diophantine type condition. Karaliolios \cite{karaliolios2018local} proved a local rigidity property for Diophantine rotations on tori under the weaker condition that at least one orbit admits a Diophantine rotation vector.

In this paper, we study local rigidity property of $\mathbb Z\ltimes_\lambda\mathbb R $ action assuming neither hyperbolicity of the toral automorphism nor preservation of rotation vector for the flow. Because of the structure of our group $\mathbb Z\ltimes_\lambda\mathbb R $, we shall expect inheritance of structural stability from Anosov diffeomeorphism and local rigidity from linear flow part and they coincide to some extent. It turns out with certain condition this group action has a weak version of local rigidity, i.e., any perturbation perpendicular to linear flow direction doesn't destroy the dynamics and any perturbation is essentially a linear time change along flow direction. It can be stated as following. 
\begin{theorem}\label{ThmMain}
Let $\alpha $  be the affine $\mathbb Z\ltimes_\lambda\mathbb R $-action on torus $\mathbb{T}^d$ generated by pair $(A, v)$, where A is an irreducible toral automorphism and $v$ is a non-zero eigenvector whose corresponding eigenvalue is $\lambda\in\mathbb R$. Then there exists $r=r(d)$ and $\delta=\delta(d,\alpha)>0$  such that if $\beta$ is a $C^\infty$-smooth $\mathbb Z\ltimes_\lambda\mathbb R $-action satisfying $d_r(\alpha, \beta)<\delta$, then there is a vector $v^*$ proportional to $v$ such that $\beta$ is $C^\infty$-conjugate to the affine action generated by $(A,v^*)$. Moreover, one may choose $r=42(d+1)$.
\end{theorem}

Here $d_r$ denotes the distance between $\alpha$ and $\beta$ in the $C^r$ topology.

We prove our theorem using a KAM mechanism. In classical KAM method, certain conditions on the preservation of rotation vectors are needed to absorb the zeroth Fourier coefficients at each step. In our proof, the Fourier coefficient of zero frequency, or at least its component in the direct sum of all the eigenspaces except the one containing $v$, can be transformed by the automorphism $A$ and absorbed thereafter.

Local rigidity of solvable actions has been investigated by other authors. In the paper \cite{burslem2004global} Burslem and Wilkinson studied the action of Baumslag-Solitar group $BS(1,k)= <a, b| aba^{-1}=b^k>$ on $\mathbb{R}P^1$ and proved a classification theorem. In Asaoka's two papers \cites{asaoka2012rigidity,Asaoka2014rigidity}, he investigated $BS(n,k)$-action on the sphere $S^n$ and torus $\mathbb{T}^n$. The recent paper \cite{wilkinson2017rigidity} by Wikinson and Xue studied ABC-group actions on tori and achieved local rigidity assuming rotation vectors are preserved as well as a simultaneous Diophantine condition on them. Note that in these papers, the rank of the flow part of the action equals the dimension of the nilmanifold, while in our case the flow part has rank 1. Our result studies a smaller group action with less group relations, and does not assume the existence of rotation vectors.

{\bf Acknowledgements.} The author would like to thank his advisor Zhiren Wang for helpful discussion and careful reading the manuscript and suggestion, his co-advisor Anatole Katok for tremendous help and support during this project. Last but not least the author would like to thank  Federico Rodriguez Hertz for pointing out the reference \cite{hamilton1982inverse} for Lemma \ref{LemCompNorm}.

\section{Preliminaries}

In this paper $\alpha$ will be an affine action on torus $\mathbb{T}^d$ by the solvable group $\mathbb Z\ltimes_\lambda \mathbb R$. The group rule is given by $$(n_1,t_1)(n_2,t_2)=(n_1+n_2, \lambda^{-n_2} t_1+t_2)$$ for $n_1,n_2\in \mathbb Z$ and $t_1, t_2 \in \mathbb R$. More precisely, $\alpha$ is a group morphism $\alpha: \mathbb Z\ltimes_\lambda\mathbb R \rightarrow  \mathrm{Diff}^\infty(\mathbb{T}^d)$ with $\alpha(n,0).x=A^n x$, $\alpha(0,t). x=x+vt$ where $A$ is an ergodic toral automorphism and $\lambda$, $v$ are respectively a real eigenvalue of $A$ and a corresponding eigenvector. Note that because $A$ is irreducible, the eigenspace $V$ corresponding to the eigenvalue $\lambda$ is $1$-dimensional.

We call $\alpha$ the affine $\mathbb Z\ltimes_\lambda\mathbb R $-action generated by the pair $(A, v)$ where $v$ represents the constant vector field that generates the linear flow.  
 
Let $\beta$ be a smooth $\mathbb Z\ltimes_\lambda\mathbb R $-action. We say $\beta$ is generated by the pair $(\tA, \tv)$, where $\tA$ is a diffeomorphism homotopic to $A$ and $\tv$ is a smooth vector field, if $\beta(1,0)=\tA$ and $\beta(0,t)=\tilde\Phi^t$ is the flow along $\tv$.

The distance between $\alpha$ and $\beta$ in $C^r$ topology can be taken to be   $$d_r(\beta,\alpha)=(||\tilde A-A||_r,||\tilde v-v||_r).$$ Here   $||\cdot||_r$ represents the general $C^r$ norm $ ||\cdot||_{C^r}$. 

Let $P_V$ be the projection into $V$ in the decomposition \begin{equation}\label{EqEigenDecomp}\mathbb R^d=V\oplus V^{\perp},\end{equation} where $V^{\perp}$ denotes the direct sum of all eigenspaces of $A$ other than $V$.

In this paper, $x \ll_z y$ denotes $ x\leq C y,$ where $C$ depends only on z.

A vector $v\in \mathbb{R}^d$ is said to satisfy a Diophantine condition of type $ (C, \tau)$, if $|\langle k,v\rangle|\ge \frac{C}{|k|^\tau}$ for all $k\in \mathbb Z^d \backslash \{0\}$.

\begin{lemma}\label{LemDio}\cite{katzenlson1971ergodic}*{Lemma 3} If $A$ is an irreducible $d\times d$ matrix of integer coefficients, and $v$ is a non-zero eigenvector of $A$ from a $1$-dimensional eigenspace, then $v$ is Diophantine of type $(C,\tau)$ where $\tau=d-1$.
\end{lemma}

To carry out the KAM scheme, we will set $\tau=d-1$ and work under the following assumption:

\begin{assumption}\label{AssuDio}The eigenvalue $v$ and the perturbation error $w=\tv-v$ satisfy: \begin{enumerate}\item $|v|\in[\frac12,2]$;
\item $\hat{w}(0)\in V^{\perp}$,  where $V^{\perp}$ is as in the decomposition
 \eqref{EqEigenDecomp}.
\end{enumerate}
\end{assumption}

\begin{remark}\label{RemDio}By Lemma \ref{LemDio}, part (1) of Assumption \ref{AssuDio} implies that $v$ is $(C,\tau)$-diophantine for some constant $C=C(d,A)$.\end{remark}

We will use  $|\cdot|_r$ for $|f|_r=\sup_{n\in \bZ }|\hat{f}(n)|\cdot| n|^r$ where ${\hat{f}(n)}$ is the $n$-th Fourier coefficient. $|\cdot|_r$ is a semi-norm. The following lemma is standard and describes the relation between $||\cdot||_r$ and $|\cdot|_r$. (See e.g. \cite{de2001tutorial}.)

\begin{lemma}\label{l1}
Let $f\in C^\infty(\mathbb{T}^d,\mathbb R^d)$ with $\hat{f}(0)=0$, then $ |f|_r\le ||f||_r$, $||f||_r\ll_{r,d} |f|_{r+d+1}$. 
\end{lemma}

When $F$, $G$ are diffeomorphisms on $\mathbb T^d$, the $C^r$ norm of their composition can be controlled linearly provided that $\|F\|_1$ and $\|G\|_1$ are bounded:

\begin{lemma}\label{LemCompNorm}\cite{hamilton1982inverse}*{Lemma 2.3.4}
Suppose $F,G\in C^r(\mathbb{T}^d,\mathbb T^d)$ and $\|F\|_1,\|G\|_1\leq M$, then $$\|F\circ G\|_r\ll_{r,d,M} 1+||F||_r+||G||_r.$$ 
\end{lemma}

Also the following interpolation inequalities will be used a lot in our proof. Suppose $f\in C^{\infty}(\mathbb{T}^d)$, then
\begin{equation}\label{EqInterpolate}
||f||_r\ll_{s,r_1,r_2,d} ||f||^{1-s}_{r_1}||f||_{r_2}^{s},
\end{equation} 
where $r=(1-s)r_1+sr_2 $ and $0\le s \le1.$

\section{Preparatory Steps}

The perturbation $\beta$ may not be conjugate to $\alpha$.  Finding the conjugacy can be transformed into solving the equation $\alpha \circ H=H \circ \beta$, or equivalently to finding $H$ that satisfies the following two equations. 
\begin{align*}
A\circ H & = H \circ \tA \\
\Phi^t \circ H & = H \circ \tilde{\Phi}^t 
\end{align*}Here $\Phi^t$ and $\tilde\Phi^t$ are respectively the smooth flows generated by $v$ and $\tv$. Differentiating with repsect to $t$ in the second equation above, we see
\begin{equation}\label{eq:1}
v\circ H = DH \cdot \tv 
\end{equation}
There is no direct way to solve this nonlinear equation \eqref{eq:1}. In the classical KAM theory of studying linear flow on torus, instead of solving it directly, one linearizes the problem first and then resolves the linearized equation, using the solution to the linearized equation to construct a new approximate solution to the non-linear problem which is much closer to being an actual solution compared with previous one. By iterating this process, a convergent sequence of solutions will be obtained and its limit is the solution to original nonlinear equation. we adapt the KAM method in our proof. 

Let $H=\mathrm{Id}+h$ and $\tilde v=v+w$, linearizing equation \eqref{eq:1} gives 
\begin{equation}\label{eq:2}
-w=Dh\cdot v
\end{equation}

\subsection{Construction of conjugacy}
Equation \eqref{eq:2},  in terms of Fourier coefficients, can be rewritten as 
\begin{equation*}
-\hat{w}(n)=\hat{h}(n) 2\pi in\cdot v 
\end{equation*} 

This equation gives the solution   
\begin{equation}
\hat{h}(n)=\frac{-\hat{w}(n)}{2\pi in\cdot v},\, \forall n\in \bZ ^d \setminus \{0\} \label{eq:3}
\end{equation}
To get $\hat{h}(0)$, let $f=\tA-A$, one can linearize $A\circ H = H \circ \tA$ to get
\begin{equation*}
A\circ h=f+h\circ A
\end{equation*}
Rewriting it in terms of the Fourier coefficient at 0 frequency, we have
\begin{equation*}
A \hat{h}(0)=\hat{f}(0)+\hat{h}((A^{-1})^{\T}(0))
\end{equation*}

Since $A$ is an ergodic toral automorphism,  there exists a unique solution
\begin{equation} 
\hat{h}(0)=(A-\mathrm{Id})^{-1}\hat{f}(0) \label{eq:4}
\end{equation} 
With all the $\hat{h}(n)$'s solved above in \eqref{eq:3} and \eqref{eq:4}, define
\begin{equation}\label{eq:5}
h(x)=\sum_{n\in \bZ ^d}\hat{h}(n)\exp(2\pi in \cdot x)
\end{equation} 
Since $\beta$ is a $\mathbb Z\ltimes_\lambda\mathbb R $-action on torus, it satisfies $\tA\circ \tilde{\Phi}^t=\tilde{\Phi}^{\lambda t} \circ \tA$. Differentiating in $t$, we get
\begin{equation*}
(A+Df)\cdot (v+w) =\lambda (v+w)\circ(A+f) 
\end{equation*}
or 
\begin{equation*}
A\cdot v+ A \cdot w +Df \cdot v + Df\cdot w =\lambda v+ \lambda w \circ (A+f)
\end{equation*}
Applying $A\cdot v=\lambda v$ and subtracting both sides by $\lambda w \circ A$,
\begin{equation}\label{eq:6}
A\cdot w + Df \cdot v - \lambda w \circ A=\lambda w\circ(A+f)-\lambda w\circ A-Df\cdot w
\end{equation}

For notational simplicity, let $E$ denote the right side of  \eqref{eq:6}, i.e., 
\begin{equation*}
E\coloneqq\lambda w\circ(A+f)-\lambda w\circ A-Df\cdot w
\end{equation*}
Its $C^0$ norm can be estimated by
\begin{equation}\label{E0}
||E||_0\le |\lambda| ||w||_1||f||_0+||f||_1||w||_0
\end{equation}

Writing equation \eqref{eq:6} in Fourier coefficients, we get
\begin{equation}\label{eq:7}
A\hat{w}(n)+\hat{f}(n)2\pi in\cdot v-\lambda \hat{w}((A^{-1})^{\T}(n))=\widehat{E}(n)
\end{equation}
When $n=0$, $A\hat{w}(0)-\lambda \hat{w}(0)=\hat{E}(0)$. This forces $\hat{E}(0)$ to lie in $V^{\perp}$. By Assumption \ref{AssuDio},
\begin{equation}
\hat{w}(0)=(A-\lambda \cdot \mathrm{Id})|_{V^{\perp}}^{-1} \hat{E}(0).
\end{equation}
When $n\neq 0$, $|2\pi in\cdot v|>0$ because of the Diophantine property of $v$. Dividing both sides of \eqref{eq:7} by  $2\pi in\cdot v$ yields, with definition of $h$ in \eqref{eq:5}, 
\begin{equation}\label{EqNon0Error}
-A(\hat{h}(n))+\hat{f}(n)+\hat{h}((A^{-1})^{\T}(n))=\frac{\hat{E}(n)}{2\pi in\cdot v}
\end{equation}
When $n = 0$, from the construction of $\hat{h}(0)$, it satisfies \begin{equation}\label{Eq0Error}A \hat{h}(0)=\hat{f}(0)+\hat{h}((A^{-1})^{\T}(0)).\end{equation} Combining these,  one has the following equality:
\begin{equation}\label{eq:11}
-A\circ h+f+h\circ A= E^{*}
\end{equation}
where $$E^{*}=\sum_{n\in\mathbb Z\backslash\{0\}}\frac{\hat{E}(n)}{2\pi in\cdot v}\exp(2\pi in \cdot x).$$  The following norm estimates for $\hat{E}(0)$ and $E^*$ are straightforward from Remark \ref{RemDio} and Lemma \ref{l1}.

\begin{equation}
||E^*||_r \ll_{r,d} |E^*|_{r+d+1} \ll |E|_{r+d+1+\tau} \ll_{r,d} ||E||_{r+d+1+\tau}\\
\end{equation} 
After all the preparation above, we state the following lemma. 

\begin{lemma}\label{l2}
Let $\alpha$ be an affine $\mathbb Z\ltimes_\lambda\mathbb R $-action on torus $\mathbb{T}^d$ generated by pair $(A, v)$ and $\beta$ be a $C^{r+d+1+\tau}$ $\mathbb Z\ltimes_\lambda\mathbb R $-action on $\mathbb{T}^d$ generated by pair $(A+f, v+w)$. Under assumption \ref{AssuDio}, there exists a solution $h$ to equation \eqref{eq:11} and
\begin{equation}\label{Eqwh}
-w+\hat{w}(0)= Dh\cdot v
\end{equation} 
Moreover \begin{equation}\label{EqwhNorm}
 ||h||_r\ll ||f||_0 + ||w-\hat{w}(0)||_{r+d+1+\tau}\leq  ||f||_0 + ||w||_{r+d+1+\tau}. \end{equation}
\end{lemma}
\begin{proof}
Let $h$ be given by \eqref{eq:3} and \eqref{eq:4}. Then $h$ satisfies  \eqref{eq:11} and \eqref{Eqwh} naturally. To estimate the norm of $h$, notice for $n \neq 0$, $|\hat{h}(n)|=|\frac{-\hat{w}(n)}{2\pi in\cdot v}|\ll_v |\hat{w}(n)|\cdot| n|^{\tau}$ because $v \in DC(C, \tau)$. Combining with Lemma \ref{l1}, we have $$||h-\hat{h}(0)||_r\ll ||w-\hat{w}(0)||_{r+d+1+\tau}.$$ Moreover, $$|\hat{h}(0)|=|(A-\mathrm{Id})^{-1}\hat{f}(0)| \ll_A  ||f||_0.$$ This proves the lemma.
\end{proof}

\subsection{Smoothing operators pair}
It is clear from the above Lemma that there is certain loss of smoothness. In other words, we can only estimate the $C^r$ norm of $h$ by  the $C^{r+d+1+\tau}$ norm of $w$. 

To overcome this obstacle, the standard method in KAM is to introduce a family of smoothing operators $S_t: C^r(T^d, \mathbb R^d) \mapsto C^\infty(T^d, \mathbb R^d)$. Instead of solving  $-w+\hat{w}(0)= Dh\cdot v$, one can solve the following equation:
\begin{equation*}
-S_tw=Dh\cdot v
\end{equation*}

Here we use a discrete version of smoothing operators. 

\begin{lemma}\label{LemSmoothing}Suppose $\Lambda\subset\bZ^d$ satisfies $\{n\in\bZ^d:0<|n|\leq M\}\subseteq\Lambda\subseteq\{n\in\bZ^d:0<|n|\leq N\}$ where $M<N$. For a sufficiently regular function $f(x)=\sum_{n\in\bZ^d}\hat{f}(n)\exp(2\pi in \cdot x)$, Let $$S_\Lambda f(x)=\sum_{n\in\Lambda}\hat{f}(n)\exp(2\pi in \cdot x)$$ $$\dot S_\Lambda f(x)=\sum_{n\in\bZ^d\backslash(\Lambda\cup\{0\})}\hat{f}(n)\exp(2\pi in \cdot x).$$

Then,
\begin{align*}
|S_\Lambda f|_{a+b} &\le N^b|f|_a \label{12}\\
||S_\Lambda f||_{a+b} &\ll_{a,b,d} N^{b+d+1}||f||_a
\end{align*}
and 
\begin{align*}
|\dot{S}_\Lambda f|_{a-b} &\le M^{-b}|f|_a\\
||\dot{S}_\Lambda f||_{a-b} &\ll_{a,b,d} M^{-b+d+1}||f||_{a}
\end{align*}
for $a\ge b \ge 0$. \end{lemma}

\begin{proof}Since $S_\Lambda f$ and $\dot{S}_\Lambda f$ both satisfy assumption in Lemma \ref{l1}, it suffices to estimate the $|\cdot|$ norms. 

\begin{align*}|S_\Lambda f|_{a+b}\leq&\sup_{0<|n|\le N} |\hat{f}(n)|\cdot| n|^{a+b}\\
\le&\sup_{0<|n|, |A^{-1}n|\le N} |\hat{f}(n)|\cdot| n|^{a}N^b\le N^{b}|f|_a.\end{align*}

Similarly,
\begin{align*}|\dot{S}_\Lambda f|_{a-b}=&\sup_{n\not\in\Lambda, n\neq 0} |\hat{f}(n)|\cdot| n|^{a-b}\le \sup_{|n|>M}|\hat{f}(n)|\cdot| n|^{a-b}\\
\le& \sup_{|n|>M}||\hat{f}(n)|\cdot| n|^{a}M^{-b}\le M^{-b}|f|_a.\end{align*}
\end{proof}

We now define a group of operators using Lemma \ref{LemSmoothing} as follows:
\begin{itemize}
\item $S_N=S_\Lambda$ and $\dot{S}_N=\dot{S}_\Lambda$, when $$\Lambda=\{n\in\bZ^d:0<|n|\leq N\};$$
\item $T^\#_N=S_\Lambda$ and $\dot{T}^\#_N=\dot{S}_\Lambda$, when $$\Lambda=\{n\in\bZ^d:0<|n|\leq N, 0<|A^\T n|\leq N\};$$
\item $T_N=S_\Lambda$ and $\dot{T}_N=\dot{S}_\Lambda$, when $$\Lambda=\{n\in\bZ^d:0<|n|\leq N, 0<|(A^\T)^{-1}n|\leq N\}.$$
\end{itemize}
Then they all satisfy the condition of Lemma \ref{LemSmoothing} with $M=\|A\|^{-1}N$.

The same proof of Lemma \ref{l2} implies there exists $h$ which solves the equation 
\begin{equation}\label{EqSNwh}-S_N w = Dh \cdot v.\end{equation} 

However, \eqref{eq:11} no longer holds. Instead, with the new operators constructed above, , we apply $T_N$ to both sides of \eqref{eq:7} to obtain
\begin{equation}\label{eq:30}
-A\circ T_N h+T_N f+(T_N^\#h)\circ A=T_N E^{*}
\end{equation}

Instead of \eqref{EqwhNorm}, we have by Lemma \ref{LemSmoothing}
\begin{equation}\label{EqSNwhNorm}
||h||_r\ll_{r,d,A} \|f\|_0+||S_Nw||_{r+d+1+\tau}\ll_{r,d} \|f\|_0+N^{2d+2+\tau}||w||_{r}.
\end{equation} We also have
\begin{equation}\label{EqSNwhNorm1}
||h||_{r+1}\ll_{r,d,A} \|f\|_0+||S_Nw||_{r+d+2+\tau}\ll_{r,d} \|f\|_0+N^{2d+3+\tau}||w||_{r}.
\end{equation}

All the above is summarized in the following lemma: 

\begin{lemma}\label{l3}

Let $\alpha$ be an affine $\mathbb Z\ltimes_\lambda\mathbb R $-action on torus $\mathbb{T}^d$ generated by pair $(A, v)$  and $\beta$ be a $C^r$ $\mathbb Z\ltimes_\lambda\mathbb R $-action on $\mathbb{T}^d$ generated by pair $(A+f, v+w)$. Under Assumption \ref{AssuDio}, there exists a solution $h$ to equation \eqref{EqSNwh}, \eqref{eq:30} and the norm estimates \eqref{EqSNwhNorm},\eqref{EqSNwhNorm1}.
\end{lemma} 

\subsection{Inductive lemma}
Let $\alpha$ be the affine $\mathbb Z\ltimes_\lambda\mathbb R $-action generated by $(A, v)$  and $\beta$ be a $C^r$ smooth perturbation  generated by $(\tA, \tv )=(A+f, v+w)$. Denote $\epsilon_k=\|f\|_k$ and $\eta_k=\|w\|_k$.

If the perturbation $\beta$ is close enough to $\alpha$ in $C^r$ with $r>d+1$, i.e. $C^r$ norms of $f$ and $w$ are small enough, with proper chosen N, $h$'s $C^r$ norm would be small by Lemma \ref{l3} and invertibility of $\mathrm{Id}+h$ is guaranteed by inverse function theorem. We now state and prove the inductive lemma. 
\begin{lemma}\label{LemInductive} Under Assumption \ref{AssuDio}, there is a positive constant $\delta$ that depends only on $d$ and $A$, suppose 
\begin{equation}\label{EqCompCond}
 \epsilon_1<1,
\end{equation} and for some $N\in\mathbb N$,
\begin{equation}\label{EqInvertCond}
 \epsilon_0+N^{2d+2+\tau}\eta_1 < \delta.
\end{equation}
Then there exists $H\in\mathrm{Diff}^{\infty}(\mathbb{T}^d)$ such that $\beta'=H\circ \beta\circ H^{-1}$ is generated by $(\tA',\tv ')$ and the following hold true for the $\mathbb Z\ltimes_\lambda\mathbb R $-action $\beta'$ thus defined. 
\begin{align*}
||\tA'-A||_0  &\ll_{r,d,A}  N^{-r+d+1}\epsilon_r+N^{-r+3d+3+\tau}\eta_r+N^{2d+2+\tau}\eta_0^{1-\frac{1}{r}}\eta_r^{\frac{1}{r}}\epsilon_0 \\ 
& +N^{d+1+\tau}\eta_0\epsilon_0^{1-\frac{1}{r}}\epsilon_r^{\frac{1}{r}}\\
||\tA'-A||_r & \ll_{r,d,A} 1+\epsilon_r+N^{2d+2+\tau}\eta_r\\
||\tv '-v||_0 &\ll_{r,d,A}  N^{-r+d+1}\eta_r+N^{2d+2+\tau}\eta_0^{2-\frac{1}{r}}\eta_r^{\frac{1}{r}}+ \eta_0^{1-\frac{1}{r}}\eta_r^{\frac{1}{r}}\epsilon_0+\eta_0\epsilon_0^{1-\frac{1}{r}}\epsilon_r^{\frac{1}{r}}\\
||\tv '-v||_r &\ll_{r,d,A} 1+N^{2d+3+\tau}\eta_r.\\
\end{align*}
\end{lemma}

\begin{proof} Let $h$ be as in Lemma \ref{l3}. With a properly chosen constant $\delta$, the condition \eqref{EqInvertCond} implies $||h||_1 <\frac12$. In this case, $H=\mathrm{Id}+h$ is invertible,  and $H^{-1}$ is also $C^r$ with $\|H^{-1}-\mathrm{Id}\|_r\ll_{r,d} \|h\|_r$ by inverse function theorem.

\begin{align*} 
&H \circ \tA\circ H^{-1}-A\\
=&H\circ(A+f)\circ H^{-1}-A \circ H \circ H^{-1} \\
=& ((\mathrm{Id}+h)\circ (A+f)- A \circ (\mathrm{Id}+h))\circ H^{-1}\\
=& (f -A\circ h+h\circ A+h\circ \tilde A-h\circ A)\circ H^{-1}\\
=& (T_N f + \dot{T}_N f+\hat{f}(0)-A\circ T_N h-A\circ \dot{T}_N h- A \circ \hat{h}(0)\\
 &+ T_N^\# h\circ A +\dot{T}_N^\# h \circ A + \hat{h}(0)+ h\circ\tilde A-h\circ A)\circ H^{-1}\\
=& (\dot{T}_N f)\circ H^{-1} - (A\circ \dot{T}_N h)\circ H^{-1} +(\dot{T}_N^\# h\circ A)\circ H^{-1}+\\
&(h\circ \tilde A-h\circ A)\circ H^{-1} + (T_N f-A\circ T_N h+T_N^\#h\circ A)\circ H^{-1}
\end{align*} where $\hat{f}(0)-A \circ \hat{h}(0)+\hat{h}(0)=0$ with defintion of $\hat{h}(0)$ in \eqref{eq:4}. 
Enumerate parentheses in the last expression above from the left to the right by $\Omega_1,\cdots,\Omega_5$. Their $C^0$ norm are bounded as below, using Lemma \ref{LemSmoothing} and Lemma \ref{l3}.
\begin{equation*}
||\Omega_1||_0 =||\dot{T}_Nf||_0\ll_{r,d} N^{-r+d+1}||f||_r
\end{equation*}
\begin{align*}
||\Omega_2||_0  &\le ||A||_1||\dot{T}_N h||_0=||A||_1||\dot{T}_N (h-\hat{h}(0))||_0\\
&\ll_{r,d,A} N^{-r+d+1}||h-\hat{h}(0)||_r \ll_{r,d}l N^{-r+3d+3+\tau}||w||_r
\end{align*}
\begin{align*}
||\Omega_3||_0 &\le||\dot{T}_N^\# h||_0=||\dot{T}_N^\# (h-\hat{h}(0))||_0\\
&\ll_{r,d} N^{-r+d+1}||h-\hat{h}(0)||_r\ll_{r,d}  N^{-r+3d+3+\tau}||w||_r\\
\end{align*}
\begin{align*}
||\Omega_4||_0 & \le ||Dh||_0||f||_0 =||D(h-\hat{h}(0))||_0||f||_0 \\
&\le ||h-\hat{h}(0)||_1||f||_0\l_{r,d,A}l  N^{2d+2+\tau}||w||_1 ||f||_0
\end{align*}
Here, as $\dot{T}_N h$ and $\dot{S}_N h$ and $Dh$ don't depend on zeroth Fourier frequency of $h$, we could replace $h$ with $h-\hat{h}(0)$.  Finally,
\begin{align*}
||\Omega_5||_0  &=||T_NE^*||_0\ll_d |T_NE^*|_{d+1}\ll_{d,A} N^{d+1+\tau}|E|_0\le N^{d+1+\tau}||E||_0 \\
&\ll N^{d+1+\tau}(||w||_1||f||_0+||f||_1||w||_0)
\end{align*}
by Lemma \ref{l1} and inequality \eqref{E0}. 
  
Combining above inequalities and applying interpolation inequalities, we have the following norm error estimates
\begin{align*}
||H\circ \tA\circ H^{-1}-A||_0 \ll &N^{-r+d+1}||f||_r+N^{-r+3d+3+\tau}||w||_r\\
&+(N^{2d+2+\tau}+N^{d+1+\tau})||w||_0^{(1-\frac{1}{r})}||w||_r^{\frac{1}{r}}||f||_0 \\
&+ N^{d+1+\tau}||w||_0||f||_0^{(1-\frac{1}{r})}||f||_{r}^{\frac{1}{r}}\\
\end{align*}

Now we estimate the $C^r$ norm error of $H \circ \tA\circ H^{-1}-A$. By \eqref{EqCompCond}, \eqref{EqInvertCond}, $\|f\|_0, \|h\|_1<1$. Moreover, $\|H^{-1}-\mathrm{Id}\|_r\ll_{r,d}\|h\|_r$. Hence by Lemma \ref{LemCompNorm}, 
\begin{align*}
&||H\circ \tA\circ H^{-1}-A||_r\\
\le &||h\circ \tA \circ H^{-1}||_r+||f\circ H^{-1}||_r+||A\circ H^{-1}-A||_r \\
\ll_{r,d,A}&1+\|f|_r+\|h\|_r\\
\ll_{r,d,A}&1+\|f\|_r+N^{2d+2+\tau}\|w\|_r
\end{align*}
In this norm estimate, error of the automorphism part depends on that of the flow part  since  Fourier coefficients of $h$ at nonzero frequencies are solved from the flow part. 

It remains to bound the error in the flow part, for which we have following decomposition:
\begin{align*}
& (DH \cdot \tv  )\circ H^{-1}-v\\
=& (\mathrm{Id}+Dh)\cdot (v+w)\circ H^{-1}-v\\
=& (w+Dh\cdot v+Dh\cdot w)\circ H^{-1}\\
= &(\hat{w}(0)+\dot{S}_N w +S_N w + Dh \cdot v + Dh\cdot w)\circ H^{-1}\\
= &\hat{w}(0)\circ H^{-1}+\dot{S}_N w\circ H^{-1} + (Dh\cdot w)\circ H^{-1} 
\end{align*}
Each term's $C^0$ norm are bounded as follows:
\begin{equation*}
||\hat{w}(0)\circ H^{-1}||_0=||(A-\lambda\mathrm{Id})^{-1}\hat{E}(0)||_0\ll_{d,A} |\hat{E}(0)|\leq ||E||_0
\end{equation*}
because of the assumption $\hat{w}(0)\in V^{\perp}$ from Assumption \ref{AssuDio}. 
\begin{equation*}
||\dot{S}_N w\circ H^{-1}||_0=||\dot{S}_N w||_0\ll_{r,d} N^{-r+d+1} ||w||_r
\end{equation*}
\begin{align*}
||(Dh\cdot w)\circ H^{-1}||_0\le &||Dh||_0||w||_0\le ||h-\hat{h}(0)||_1||w||_0\\
\ll_{d,A} &N^{2d+2+\tau}||w||_1||w||_0
\end{align*}
 Combining them and applying interpolation inequalities, it can be improved into the following. 
\begin{align*}
||DH \cdot \tv  \circ H^{-1}-v||_0 & \ll N^{-r+d+1}||w||_r+N^{2d+2+\tau}||w||^{(2-\frac{1}{r})}_0||w||_r^{\frac{1}{r}}\\
& + ||w||_0^{(1-\frac{1}{r})}||w||_r^{\frac{1}{r}}||f||_0+ ||w||_0||f||_0^{(1-\frac{1}{r})}||f||_r^{\frac{1}{r}}, 
\end{align*}

Finally, we estimate the $C^r$ norm error using Lemma \ref{LemCompNorm},  \eqref{EqSNwhNorm}, \eqref{EqSNwhNorm1} and the fact that $\|w\|_1,\|h\|_1<1$:
\begin{align*}
 & ||DH\cdot \tv \circ H^{-1}-v||_r\\
\ll_{r,d,A}&1+||DH||_r+||\tv||_r+|v|+\|H^{-1}\|_r\\
\ll_{r,d,A}&1+||h||_{r+1}+\|w\|_r\\
\ll_{r,d,A}&1+ N^{2d+3+\tau}||w||_r+||w||_r\\
\ll_{r,d,A}&1+ N^{2d+3+\tau}||w||_r
\end{align*}
By now we have proved the inductive lemma.
\end{proof}

\section{Proof of Theorem \ref{ThmMain}}

We now prove Theorem \ref{ThmMain}. Let $r\in\mathbb N$ be a parameter that will be determined later. For simplicity, we will omit the subscripts $_{r,d,A}$ and write $\ll$ for $\ll_{r,d,A}$ from now on.

In our proof, we shall not linearize $(\tA,\tv )$ around $(A,v)$. This is because, as we don't assume preservation of rotation vector, the flow part is not guaranteed to be  conjugate to the original linear flow, and one has to allow a linear time change. To overcome this, in every step of the KAM scheme, one will linearize around a different affine action, which is obtained by a linear time change. This will not affect the condition (1) from Assumption \ref{AssuDio}, and the linear time change will be chosen so that condition (2) from Assumption \ref{AssuDio} remains true.

To begin the process,  we set up $\tA_0=\tA$, $f_0=f$, $\tv_0=\tv$, $v_0=v$, $\beta_0=\beta$ and $H_0=\mathrm{Id}$. It will be assumed that this configuration satisfies Assumption \ref{AssuDio} as well as the inequality \eqref{EqInvertCond}.

In the $n$-th step, given a smooth action $\beta_n$ generated by $\tA_n=A+f_n$ and $\tv_n=v_n+w_n$. Suppose $v_n$ and $w_n$ satisfies Assumption \ref{AssuDio}, and the inequality \eqref{EqInvertCond} with a properly chosen large integer $N_n$, then one can apply Lemma \ref{LemInductive} to obtain a new smooth action $\beta_{n+1}$, generated by $\tA_{n+1}=A_n+f'_n$ and $\tv_{n+1}=v_n+w'_n$, and a conjugacy $H_{n+1}=\mathrm{Id}+h_{n+1}$ between $\beta_{n+1}$ and $\beta_n$. To insure Assumption \ref{AssuDio} in the following step, we define
\begin{equation}f_{n+1}=f'_n, v_{n+1}=v_n+P_V(\widehat{w'_n}(0)), w_{n+1}=w'_n-P_V(\widehat{w'_n}(0)).\end{equation}  Then $\tA_{n+1}=A_n+f_{n+1}$, $\tv_{n+1}=v_{n+1}+w_{n+1}$. Moreover, $v_{n+1}$ is a vector from $V$, and $w_{n+1}$ satisfies condition (2) in Assumption \ref{AssuDio}.

It should be remarked that, because $w_{n+1}$ and $w'_n$ differ by the constant vector $P_V(\widehat{w'_n}(0))$, which satisfies $|P_V(\widehat{w'_n}(0))|\ll\|w'_n\|_0$,
\begin{equation}\label{EqDropProj}\|w_{n+1}\|_l\leq \|w'_n\|_l, \forall l\geq 0.\end{equation}

In the rest of the proof, we will write $\epsilon_{n,l}=||f_n||_l$ and $\eta_{n,l}=||w_n||_l$ for all $l\geq 0$ once $f_n$ and $w_n$ are defined.

In order to be able to iterate the inductive lemma, it remains to verify: condition (1) in Assumption \ref{AssuDio} for $v_{n+1}$, namely that \begin{equation}\label{EqvLengthNext}|v_{n+1}|\in[\frac12,2];\end{equation} and the inequalities \eqref{EqCompCond}, \eqref{EqInvertCond} for $f_{n+1}$, $w_{n+1}$ and a properly chosen integer $N_{n+1}$, namely that \begin{equation}\label{EqCompCondNext}\epsilon_{n+1,1}<1.\end{equation}\begin{equation}\label{EqInvertCondNext}\epsilon_{n+1,0}+N^{2d+3+\tau}\eta_{n+1,1}<\delta.\end{equation}

The sequence $\{N_n\}$ is determined as follows: let $N_0=N_0(r,d,A)$ is a sufficiently large integer which in particular makes \eqref{EqInvertCond} hold for $\epsilon_{0,0}$ and $\eta_{0,1}$. We then fix some $\sigma\in(0,1)$, say $\sigma=\frac12$, and define $N_n$ inductively by $N_{n+1}=N_n^{1+\sigma}$.  With these choices, \eqref{EqvLengthNext}, \eqref{EqCompCondNext}, and \eqref{EqInvertCondNext} will be guaranteed by the following lemmas.

\begin{lemma}\label{LemCrBound}
For all $r\ge 0$ and $k>\frac{2d+3+\tau}{\sigma}$, if $\epsilon_{n,r}\ll_r N_n^k$ and $\eta_{n,r}\ll_r N_n^k$, then $\epsilon_{n+1,r}\ll N_{n+1}^k$ and $\eta_{n+1,r} \ll N_{n+1}^k$.\end{lemma}
\begin{proof}
From Lemma \ref{LemInductive}, one deduce
\begin{equation*}
\epsilon_{r,n+1}\ll \|f_n\|_r\ll 1+N_n^{k}+N_n^{2d+2+\tau+k}\ll  N_n^{2d+2+\tau+k} \ll (N_{n}^{1+\sigma})^k\ll N_{n+1}^{k}.
\end{equation*}
Using \eqref{EqDropProj}, we also get
\begin{equation*}
\eta_{r,n+1}\ll \|w'_n\|_r\ll 1+N_n^{2d+3+\tau+k}\ll N_n^{2d+3+\tau+k}\ll (N_{n}^{1+\sigma})^k\ll N_{n+1}^{k}.
\end{equation*}
This completes the proof of the lemma. \end{proof}

We then prove that $\epsilon_{n+1,0}$ and $\eta_{n+1,0}$ are small.

\begin{lemma}\label{LemC0Bound} Given $\sigma$ and $k$ as in Lemma \ref{LemCrBound}. If $$r>\max\Big(\frac 2{1-\sigma}, (1+\sigma)\big(3d+3+\tau+2k+\frac{2(2d+2+\tau)}{1-\sigma}\big)\Big),$$ then there exists a constant $y>0$ such that:

If, in addition to the conditions in Lemma \ref{LemCrBound}, $\epsilon_{n,0}\ll N_n^{-y}$ and $\eta_{n,0}\ll N_n^{-y}$, then $\epsilon_{n+1,0}\ll N_{n+1}^{-y}$ and $\eta_{n+1,0}\ll N_{n+1}^{-y}$.
\end{lemma}

\begin{proof}
By Lemma \ref{LemInductive}, 
\begin{align*}
\epsilon_{n+1,0}  \ll & N_n^{-r+d+1}\epsilon_{r,n}+N_n^{-r+3d+3+\tau}\eta_{r,n}+N_n^{2d+2+\tau}\eta_{0,n}^{1-\frac{1}{r}}\eta_{r,n}^{\frac{1}{r}}\epsilon_{0,n}\\
&+N_n^{d+1+\tau}\eta_{0,n}\epsilon_{0,n}^{1-\frac{1}{r}}\epsilon_{r,n}^{\frac{1}{r}}\\
\eta_{n+1,0} \ll & \|w'_n\|_0\\
\ll &N_n^{-r+d+1}\eta_{r,n}+N^{2d+2+\tau}\eta_{0,n}^{2-\frac{1}{r}}\eta_{r,n}^{\frac{1}{r}}+ \eta_{0,n}^{1-\frac{1}{r}}\eta_{r,n}^{\frac{1}{r}}\epsilon_{0,n}+\eta_{0,n}\epsilon_{0,n}^{1-\frac{1}{r}}\epsilon_{r,n}^{\frac{1}{r}}
\end{align*}
By our assumptions, 
\begin{align*}
\epsilon_{0,n+1}  \ll & N_n^{-r+d+1+k}+N_n^{-r+3d+3+\tau+k}+N_n^{2d+2+\tau+\frac{k}{r}}\eta_{0,n}^{1-\frac{1}{r}}\epsilon_{0,n}\\
&+N_n^{d+1+\tau+\frac{k}{r}}\eta_{0,n}\epsilon_{0,n}^{1-\frac{1}{r}}\\
\ll & N_n^{-r+3d+3+\tau+k}+N_n^{2d+2+\tau+\frac{k}{r}-(2-\frac{1}{r})y}+N_n^{d+1+\tau+\frac{k}{r}-(2-\frac{1}{r})y}\\
\eta_{0,n+1} \ll & N_n^{-r+d+1+k}+N^{2d+2+\tau+\frac{k}{r}}\eta_{0,n}^{2-\frac{1}{r}}+ N_n^{\frac{k}{r}}\eta_{0,n}^{1-\frac{1}{r}}\epsilon_{0,n}\\
\ll & N_n^{-r+d+1+k}+N_n^{2d+2+\tau+\frac{k}{r}-(2-\frac{1}{r})y}+N_n^{\frac{k}{r}-(2-\frac{1}{r})y}
\end{align*}
To get $\epsilon_{0,n+1}\le N_{n+1}^{-y}$ and $\eta_{0,n+1}\ll N_{n+1}^{-y}$, following inequalities must hold. 
\begin{align}
-r+3d+3+\tau+k\le & -(1+\sigma)y\label{21}\\
2d+2+\tau+\frac{k}{r}-(2-\frac{1}{r})y \le & -(1+\sigma)y\label{22}\\
d+1+\tau+\frac{k}{r}-(2-\frac{1}{r})y \le &-(1+\sigma)y\label{23}\\
-r+d+1+k \le &-(1+\sigma)y\label{24}\\
\frac{k}{r}-(2-\frac{1}{r})y\le &-(1+\sigma)y\label{25}
\end{align}
Obviously, $\eqref{21}\Rightarrow \eqref{24}$ and $\eqref{22}\Rightarrow \eqref{23}\Rightarrow \eqref{25}$. We just need to solve the inequalities \eqref{21} and \eqref{22}, which can be reduced to
\begin{equation}\label{EqChooseY}
 \frac{2d+2+\tau+\frac{k}{r}}{1-\sigma-\frac{1}{r}}\le y\le  \frac{r-3d-3-\tau-k}{1+\sigma}
\end{equation}
For our choice of $r$, $\frac{r-3d-3-\tau-k}{1+\sigma}>\frac{2d+2+\tau+\frac{k}{r}}{1-\sigma-\frac{1}{r}}>0$ always hold. Therefore such a parameter $y$ exist. \end{proof}

\begin{corollary}\label{CorC1Bound} If $N_0$ is chosen to be sufficiently large, in the settings of Lemma \ref{LemCrBound} and Lemma \ref{LemC0Bound}, the inequalities \eqref{EqCompCondNext} and \eqref{EqInvertCondNext}  are both true.\end{corollary}

\begin{proof} By Lemma \ref{LemCrBound}, Lemma \ref{LemC0Bound}, interpolation inequalities, and \eqref{22},
\begin{equation}\label{EqAC1Bound}\epsilon_{n+1,1}\ll \epsilon_{n+1,0}^{1-\frac1r}\epsilon_{n+1,r}^\frac1r\ll N_{n+1}^{-(1-\frac1r)y+\frac kr}\leq N_{n+1}^{(-(2d+2+\tau)-\sigma)y};\end{equation}
\begin{equation}\label{EqvC1Bound}\eta_{n+1,1}\ll N_{n+1}^{-\sigma y}.\end{equation}

As $N_{n+1}>1$, \eqref{EqCompCond} follows immediately. For \eqref{EqInvertCond}, notice
\begin{align*}&\epsilon_{n+1,0}+N_{n+1}^{2d+2+\tau}\eta_{n+1,1}\\
\ll& N_{n+1}^{-y}+N_{n+1}^{2d+2+\tau+\frac{k}{r}-(1-\frac{1}{r})y}\ll N_{n+1}^{-y}+ N_{n+1}^{-\sigma y}.\end{align*} As $\delta=\delta(r,d,A)$ is given and $N_{n+1}>N_0$, this implies \eqref{EqInvertCond} when $N_0$ is chosen to be sufficiently large. \end{proof}

\begin{remark}Since $\tau=d-1$, with $\sigma=\frac12$ one may choose $k=6d+6$ in Lemma \ref{LemCrBound}. It is then easy to verify that $r=42(d+1)$ is sufficient for Lemma \ref{LemC0Bound}. \end{remark}

It remains to establish \eqref{EqvLengthNext}.

\begin{corollary}\label{CorvLength} When $N_0$ is chosen to be sufficiently large, $|v_0|=1$ and the conditions in Lemma \ref{LemCrBound} and Lemma \ref{LemC0Bound} hold in the $m$-th step for every $m$ from $0$ to $n$, then the inequality \eqref{EqvLengthNext} is true.\end{corollary}
\begin{proof}Using the assumptions, we know for every $m$ between $0$ and $n$, 
\begin{align*}|v_{m+1}-v_m|=&|P_V\widehat{w'_m(0)}|\ll \|w'_m(0)\|_0.\end{align*} Remark that the proofs of Lemma \ref{LemCrBound} and Lemma \ref{LemC0Bound} also applies to $\|w'_m\|_r$ and $\|w'_m(0)\|_0$ instead of $\epsilon_{m,r}$ and $\epsilon_{m,0}$, and thus $\|w'_m(0)\|\ll N_{m+1}^{-y}$. It follows that $$|v_{n+1}|-1=|v_{n+1}|-|v_0|\ll\sum_{m=0}^nN_{m+1}^{-y}=\sum_{m=0}^n\big(N_0^{((1+\sigma)^{m+1})}\big)^{-y}\ll_{\sigma,y} N_0^{-y}.$$ Once $\sigma$ and $y$ are given, one may choose a sufficiently large $N_0$, which now depends only on $r$, $d$ and $A$, such that $|v_{n+1}|-1\leq\frac12$.\end{proof}

What we have proved so far can now be summarized into the following:

\begin{theorem}\label{ThmMainC1}There exists $k$, $y$, $N_0$ and $\delta_0$, which depend only on $d$ and $A$, such that, with $r=42(d+1)$ and $\sigma=\frac12$, if $|v_0|=1$, $\|f_0\|_r<\delta_0$, $\|w_0\|_r<\delta_0$, then the inductive construction above can be iterated for all $n\geq 0$, and Lemma \ref{LemCrBound} and Lemma \ref{LemC0Bound} can be applied in every step. 

Furthermore, the conjugacy $\tilde H_n=H_n\circ H_{n-1}\circ \cdots \circ H_0$ converges in $C^1$ topology to a diffeomorphism $H:\mathbb T^d\to\mathbb T^d$, which satisfies $$H\circ \tA \circ H^{-1}=A,\ \ DH\cdot\tv\circ H^{-1}=v^*,$$ where $v^*$ is a vector proportional to $v$ and $|v^*|\in[\frac12,2]$.  \end{theorem}

\begin{proof}We may choose $\delta_0$ to reflect the initial conditions $\epsilon_{0,0}\ll N_0^{-y}$, $\eta_{0,0}\ll N_0^{-y}$, $\epsilon_{0,r}\ll N_0^k$ and $\eta_{0,r}\ll N_0^k$. Then the validity of the inductive construction, and the applicability of the lemmas in all steps, follow from Corollaries \ref{CorC1Bound} and \ref{CorvLength} .

The $C^1$ covergence of $\tilde H_n$ to a diffeomorphism $H$ is deduced from the bound \begin{align*}\|H_n-\mathrm{Id}\|_1=&\|h\|_1\ll\epsilon_{n+1,0}+N_{n+1}^{2d+2+\tau}\eta_{n+1,1}\\
\ll& N_{n+1}^{-y}+N_{n+1}^{2d+2+\tau+\frac{k}{r}-(1-\frac{1}{r})y}\ll N_{n+1}^{-y}+ N_{n+1}^{-\sigma y}\end{align*} and the fact that $N_n=N_0^{((1+\sigma)^m)}$ is fast growing. As long as $N_0$ is large enough, $\|H-\mathrm{Id}\|_1<1$ and thus $H$ is invertible. 

Finally, by \eqref{EqAC1Bound} and \eqref{EqvC1Bound}, as $n\to\infty$, $$\|\tilde H_n\circ \tA\circ \tilde H_n^{-1}-A\|_1=\epsilon_{n+1,1}\to 0.$$ Hence $H\circ \tA\circ H^{-1}=A$.

On the other hand, it follows from the proof of Corollary \ref{CorvLength} that $\{v_n\}$ is a Cauchy sequence, thus $v^*=\lim_{n\to\infty}v_n\in V$ exists and satisfies $|v^*|\in[\frac 12,2]$. Since $$\|D\tilde H_n\cdot \tv\circ \tilde H_n^{-1}-v_{n+1}\|_0=\eta_{n+1,0}\to 0,$$ we obtain $DH\cdot v\circ H^{-1}=v^*$. \end{proof}

The last step in establishing Theorem \ref{ThmMain} is to upgrade the regularity of $H$ from $C^1$ to $C^\infty$.

\begin{proof}[Proof of Theorem \ref{ThmMain}] By passing to a linear time change of the flow part, one may assume without loss of generality that $v_0=v$ has length $1$. Hence Theorem \ref{ThmMainC1} applies.

The proof of the $C^\infty$ convergence of  the sequence $\{\tilde H_n\}$ from the proof of Theorem \ref{ThmMainC1} to $H$ is standard. Indeed, by Lemma \ref{LemInductive}, for all $p>0$, there exists $U_p>0$ that depends on $p$, $d$, $A$, such that
$$1+\eta_{n+1,p}\leq U_p(1+N_n^{2d+3+\tau}\eta_{n,p}).$$ In consequence, there exists $n_0=n_0(p,d,A)$, such that for all $n\geq n_0$,
\begin{equation}\label{EqCpBound}\begin{aligned}\eta_{n,p}\leq&U_p^n\Big(\prod_{i=0}^{n-1}N_i^{2d+3+\tau}\Big)(1+\eta_{0,p})\\
=&U_p^nN_0^{(2d+3+\tau)\sum_{i=0}^{n-1}(1+\sigma)^i}(1+\eta_{0,p})\\
=&U_p^nN_0^{(2d+3+\tau)\frac{(1+\sigma)^n-1}\sigma}(1+\eta_{0,p})\\
\leq& N_0^{(2d+3+\tau)\frac{(1+\sigma)^n}{2\sigma}}(1+\eta_{0,p})\\
=&N_n^{\frac{2d+3+\tau}{2\sigma}}(1+\eta_{0,p})\end{aligned}\end{equation}

Choose $\mu\in(0,1)$ that is sufficiently close to $0$, so that $$y_\mu:=-\mu{\frac{2d+3+\tau}{2\sigma}}+(1-\mu)y>0.$$
Then by interpolation between \eqref{EqCpBound} and Lemma \ref{LemC0Bound}, for every $q\leq \mu p$ and sufficiently large $n$, 
\begin{equation}\label{EqCqBound}\begin{aligned}\eta_{n,q}\ll_{p,q}&\eta_{n,0}^{1-\frac qp}\eta_{n,p}^{\frac qp}\\
\ll_{p,q,d,A}&N_n^{-(1-\frac qp)\cdot y}N_n^{\frac qp\cdot \frac{2d+3+\tau}{2\sigma}}(1+\eta_{0,p})^{\frac qp}\\
\leq& N_n^{-y_\mu}(1+\eta_{0,p})^\mu.\end{aligned}\end{equation}

Combining \eqref{EqCqBound} with \eqref{EqSNwhNorm} and Lemma \ref{l1}, we obtain that for $m<q-(3d+3+\tau)$,
\begin{equation}\begin{aligned}\|H_n-\mathrm{Id}\|_m=&\|h_n\|_m\ll_{q,d,A}\epsilon_{n,0}+N_n^{3d+3+\tau}\eta_{n,m}\\
\ll_{q,d,A}&\epsilon_{n,0}+\eta_{n,q}\ll_{p,q,d,A}N_n^{-y}+N_n^{-y_\mu}(1+\eta_{0,p})^\mu.\end{aligned}\end{equation}

Since $N_n=N_0^{(1+\sigma)^n}$ is fast growing as $n\to\infty$, and $y>0$, $y_\mu>0$, the bound above leads to the convergence of $\tilde H^n$ to $H$ in $C^m$. By letting $p\to \infty$, $q$ and $m$ can be arbitrarily large. Therefore the convergence holds in $C^\infty$.
\end{proof}

\bibliography{citation} 
\bibliographystyle{plain}

\end{document}